\documentclass[a4paper,reqno]{amsart}
\usepackage{amssymb, amsmath, amscd}
\usepackage[final]{graphicx}
\usepackage{float}\usepackage{wrapfig}
\usepackage{color}
\usepackage{bbm}
\usepackage[final]{graphicx}
\usepackage{tikz-cd}
\usepackage{color}

\def\black{\color{black}}
\definecolor{brown(traditional)}{rgb}{0.59, 0.29, 0.0}

\def\Tr{\operatorname{Tr}}\def\Hess{\operatorname{Hess}}
\def\rr{{\mathfrak{r}}}

\newtheorem{theorem}{Theorem}[section]

\newtheorem{lemma}[theorem]{Lemma}
\newtheorem{remark}[theorem]{Remark}

\makeatletter
 \@addtoreset{equation}{section}
\makeatother
\begin{document}

\title{Centrally Harmonic spaces}
\author{P. B. Gilkey and J. H. Park}
\address{PG: Mathematics Department, University of Oregon, Eugene OR 97403-1222, USA}
\email{gilkey@uoregon.edu}
\address{JHP: Department of
Mathematics, Sungkyunkwan University, Suwon, 16419 Korea.}
\email{parkj@skku.edu}
\begin{abstract}We construct examples of centrally harmonic spaces by generalizing work of Copson and Ruse. We
show that these examples are generically not centrally harmonic at other points. We use this construction to exhibit manifolds
which are not conformally flat but such that their density function agrees with Euclidean space.
\end{abstract}


\maketitle
\section{Introduction}
\subsection{Notational conventions}
Any 2-dimensional manifold is Einstein; thus this condition imposes no additional
restrictions and the case $m=2$ is often exceptional.
We shall therefore sometimes assume
that $m\ge3$ to simplify the analysis. If $\vec\xi=(\xi^1,\dots,\xi^m)\in\mathbb{R}^m$, set:
$$\begin{array}{ll}
\|\xi\|^2:=(\xi^1)^2+\dots+(\xi^m)^2,&d\xi:=d\xi^1\dots d\xi^m,\\[0.05in]
g_e:=(d\xi^1)^2+\dots+(d\xi^m)^2,&\Delta_e^0:=-\partial_{\xi^1}^2-\dots-\partial_{\xi^m}^2,\\
S_0(r):=\{\xi:\|\xi\|=r\}.
\end{array}$$
There is a radial solution to the equation $\Delta^0_ef = 0$ for $\|\xi\|>0$ given by
$$
f(\xi):=\left\{\begin{array}{lll}
\log \|\xi\|^2&\text{if}&m=2\\[0.05in]
\|\xi\|^{2-m}&\text{if}&m > 2
\end{array}\right\}\,.
$$
Ruse \cite{R31} was the first to examine radial solutions to the Laplace equation in the { more general} context
of a Riemannian manifold $\mathcal{M}=(M,g)$ of dimension $m\ge2$.
{ Let ${\Delta_{\mathcal{M}}^0}$ (resp. ${\Delta_{\mathcal{M}}^1}$) be the Laplace-Beltrami
operator on functions (resp. 1-forms)}.
Let $r_P(Q)$ be the geodesic distance from a point $P$ to another point $Q$ of $M$. A function $ f$ is said to be \emph{radial} if {$f(Q)=\check f(r_P(Q))$ for some function $\check f$ of a single variable; in
the interests of notational simplification, we shall identify $f$ with $\check f$ when no confusion
is likely to result}. Let $\iota_P$ be the injectivity
radius. If there exists a non-constant radial function so that
${{\Delta_{\mathcal{M}}^0}} f=0$ for $0<r<\iota_P$, then $\mathcal{M}$ is said to be
\emph{centrally harmonic about} $P$. If $\mathcal{M}$ is
{centrally}  harmonic {about} every point, then $\mathcal{M}$ is said to be
a \emph{harmonic space} (see Willmore~\cite{W}).

Much of the subsequent work in the field has focussed on harmonic spaces. But in this note, we will
go back to the original question and study spaces which are centrally harmonic {about a point} $P$.
There are a number of useful characterizations of this property.
Let $(\xi^1,\dots,\xi^m)$ be geodesic coordinates
centered at a point $P$ of $M$. Such coordinate systems are characterized by the fact that the curves
$t\rightarrow t\xi$ are unit speed geodesics from $P$ if $\|\xi\|=1$ and hence $r_P(\xi)=\|\xi\|$ if
$\|\xi\|<\iota_P$.  {The} Riemannian measure
defined by $g$ {is} $\tilde\Theta_Pd\xi$ where
$\tilde\Theta_P:=\sqrt{\det{g_{ij}}}$ is the associated {\it volume density function}.
Let $S_P(r):=\{\xi\in T_PM:\|\xi\|=r\}$ be the geodesic sphere of radius $r$ centered at $P$ and
let $(r,\theta)\rightarrow r\cdot\theta$ define geodesic polar coordinates where $\theta\in S_P(1)$ and $0<r<\iota_P$.
If $d\theta$ is the Euclidean volume element of $ S_P(1)$, then $d\xi=r^{m-1}drd\theta$ so
the volume density in geodesic polar
coordinates is given by $\Theta_P:=r^{m-1}\tilde\Theta_P$.
Let
$$\Xi_P:=\partial_r\log(\Theta_P(\xi))\,;
$$
$\Xi_P$ is
the mean curvature of the geodesic sphere $S_P(\|\xi\|)$ at $\xi\in T_PM$.
For $\lambda\in\mathbb{C}$, let $\mathfrak{E}_P^0(\lambda)$ (resp. $\mathfrak{E}_P^1(\lambda)$) be the
eigen-space of radial functions (resp. 1-forms) defined by $\lambda$, i.e.,
\begin{eqnarray*}
&&\mathfrak{E}_P^0(\lambda):=\left\{\phi^0\in C^\infty(0,\iota_P):
{\Delta_{\mathcal{M}}^0}(\phi^0(r))=\lambda\phi^0(r)\right\},\\
&&\mathfrak{E}_P^1(\lambda):=\left\{\phi^1dr\in C^\infty(0,\iota_P)dr:
{\Delta_{\mathcal{M}}^1}(\phi^1(r)dr)=\lambda\phi^1(r)dr\right\}\,.
\end{eqnarray*}
{Note that we exclude the origin $r=0$; these functions are permitted to be singular at the center point $P$.}
If $\lambda\ne0$, then $d$ is an
isomorphism from $\mathfrak{E}_P^0(\lambda)$ to $\mathfrak{E}_P^1(\lambda)$ so it suffices to study
$\mathfrak{E}_P^0(\lambda)$ {in this instance}. Let $\mathfrak{H}_P^0$ (resp. $\mathfrak{H}_P^1$)
be the space of radial harmonic functions (resp. 1-forms), i.e., $\mathfrak{H}_P^0:=\mathfrak{E}_P^0(0)$ and
$\mathfrak{H}_P^1:=\mathfrak{E}_P^1(0)$.
\subsection{ Characterizations of centrally harmonic spaces} The following result was established by the authors
previously~\cite{GP20}.

\begin{theorem}\rm
The following assertions are equivalent and if any is satisfied, then $\mathcal{M}$ is
centrally harmonic about the point $P$. {{If they hold at every point, then $\mathcal{M}$ is said to be a harmonic space.}}
Let $\lambda\ne0$.
\smallbreak\noindent\centerline{\rm\begin{tabular}{lll}
(a) $\Theta_P$ is radial.&(b) $\tilde\Theta_P$ is radial.&(c) $\Xi_P$ is radial.\\
(d) $\dim\{\mathfrak{H}_P^0\}=2$.&(e) $\dim\{\mathfrak{H}_P^0\}\ge2$.&(f) $\dim\{{\mathfrak{H}_P^1}\}=2$.\\
(g) $\dim\{{\mathfrak{H}_P^1}\}\ge1.$&(h) $\dim\{\mathfrak{E}_P^0(\lambda)\}=2$.&
(i) $\dim\{\mathfrak{E}_P^0(\lambda)\}\ge1$.\\
(j) $\dim\{\mathfrak{E}_P^1(\lambda)\}=2$.&
(k) $\dim\{\mathfrak{E}_P^1(\lambda)\}\ge1$.&(l) $\Delta_{\mathcal{M}}^0 r$ is radial.
\end{tabular}}
\par\hspace{.6cm}  {\rm (m) Geodesic spheres about $P$ have
constant mean curvature.}
\end{theorem}

\subsection{Asymptotic expansion of the volume density function in geodesic coordinates} If $\mathcal{M}$ a
Riemannian manifold, then we can expand
\begin{equation}\label{E1.a}
\Theta_P(\xi)\sim\|\xi\|^{m-1}\left(1+\sum_{k=2}^\infty{{\mathcal{H}_k(\xi)}}\right)
\end{equation}
in a power series about the origin where
$\mathcal{H}_k(c\xi)=c^k\mathcal{H}_k(\xi)$ for $c\in\mathbb{R}$;
{we omit the dependence on the point $P$ in the interests of notational simplification}. Let {$R$ be the curvature
tensor of $M$ and let} $\mathcal{J}_k(\xi)$ be the
endomorphism of $T_PM$ defined by the identity
$$
g(\mathcal{J}_k(\xi)\eta_1,\eta_2)={(\nabla^kR)}(\eta_1,\xi,\xi,\eta_2;\xi,\dots,\xi)\,;
$$
$\mathcal{J}_0(\xi)$ is the Jacobi operator and $\mathcal{J}_k(\xi)={(\nabla_\xi^k\mathcal{J}_0)}(\xi)$.
We have, see for example the discussion on page 229 of \cite{W}, that
\begin{equation}\label {E1.c}\begin{array}{ll}
\mathcal{H}_2(\xi)=-\frac{\Tr\{\mathcal{J}(\xi)\}}{6},\\[0.05in]
\mathcal{H}_3(\xi)=-\frac{\Tr\{\mathcal{J}_1(\xi)\}}{12},\\[0.05in]
\mathcal{H}_4(\xi)=\frac{\Tr\{\mathcal{J}(\xi)\}^2}{72}
-\frac{\Tr\{\mathcal{J}(\xi)^2\}}{180}-\frac{\Tr\{\mathcal{J}_2(\xi)\}}{40},\\[0.05in]
\mathcal{H}_5(\xi)=\frac{\Tr\{\mathcal{J}(\xi)\} \Tr\{\mathcal{J}_1(\xi)\}}{72}
 -\frac{\Tr\{\mathcal{J}(\xi)\mathcal{J}_1(\xi)\}}{180}-\frac{\Tr\{\mathcal{J}_3(\xi)\}}{180},\\
\mathcal{H}_6(\xi)=-\frac{\Tr\{\mathcal{J}(\xi)\}^3}{1296}
 +\frac{\Tr\{\mathcal{J}(\xi)\} \Tr\{\mathcal{J}(\xi)^2\}}{1080}+\frac{\Tr\{\mathcal{J}(\xi)\} \Tr\{\mathcal{J}_2(\xi)\}}{240}
 -\frac{\Tr\{\mathcal{J}(\xi)^3\}}{2835}\\\qquad\qquad
 -\frac{\Tr\{\mathcal{J}(\xi)\mathcal{J}_2(\xi)\}}{630}+\frac{\Tr\{\mathcal{J}_1(\xi)\}^2}{288}
 -\frac{\Tr\{\mathcal{J}_1(\xi)^2\}}{672}-\frac{\Tr\{\mathcal{J}_4(\xi)\}}{1008}.
 \end{array}\end{equation}
Formulas for $\mathcal{H}_7$ and $\mathcal{H}_8$ were derived in \cite{GP20a}.
More generally, one can show that
$$
\mathcal{H}_n(\xi)=c_n\operatorname{Tr}\{\mathcal{J}_{n-2}(\xi)\}+\text{lower order terms}
$$
In particular, $c_2=-\frac16$, $c_3=-\frac1{12}$, $c_4=-\frac1{40}$,
$c_5=-\frac1{180}$,  and $c_6=-\frac1{1008}$.
We will establish
the following result in Section~\ref{S2}; it provides a leading term analysis which will be crucial in what follows.
\begin{lemma}\label{L1.2}\rm
We have $\displaystyle c_n=-\frac{n-1}{(n+1)!}$.
\end{lemma}

\subsection{ Examples of harmonic spaces}
If $\mathcal{M}$ is a simply connected 2-point homogeneous space, i.e.,
if $\mathcal{M}$ is either $\mathbb{R}^m$ or {$\mathcal{M}$}
is a rank one symmetric space, then the isometry group of $\mathcal{M}$ acts transitively on the set of unit
tangent vectors and hence $\Theta_P$ is radial for any $P$; consequently any 2-point homogeneous space
is centrally harmonic about any point and hence
is a harmonic space. In negative curvature, the Damek-Ricci spaces are also harmonic spaces;
these are solvmanifolds, but need not be 2-point homogeneous spaces.
All known harmonic spaces are locally homogeneous and modeled on one of these geometries.
{We refer to Berndt~et. al.~\cite{BTV78} for further details.}

\subsection{ Constructing centrally harmonic spaces}
Copson and Ruse~\cite{CR40} gave examples of centrally harmonic spaces by noting that a radial conformal
deformation of the {Euclidean metric $g_e$} is centrally harmonic about the origin.
{More generally, if $\mathcal{M}=(M,g)$ is the germ of a Riemannian
manifold, and if $\psi$ is a smooth non-zero function of 1-variable, we define a radial conformal deformation of
$\mathcal{M}$ by setting}
$$
\mathcal{M}_\psi:=(M,\psi(r_P^2)^{-2}g)\,.
$$
In Section~\ref{S3}, we will establish the following result which shows if $\mathcal{M}$ is centrally
harmonic about $P$, then $\mathcal{M}_\psi$ is centrally harmonic about $P$ as well. Since we can always
take the base manifold $\mathcal{M}$ to be a harmonic space,
this permits us to construct many centrally harmonic spaces generalizing the examples of Copson and Ruse~\cite{CR40}.

\begin{theorem}\label{T1.3}\rm
If $\mathcal{M}=(M,g)$ is the germ of a Riemannian manifold which centrally harmonic {about} $P$,
then $\mathcal{M}_\psi$ is centrally harmonic {about} $P$ {as well.}
\end{theorem}

\subsection{Space forms}\label{S1.6}
$\mathcal{M}$ is said to be a {\it space form} if $\mathcal{M}$
has constant sectional curvature. Let $\mathcal U$ be an open subset of $\mathbb{R}$,
let $\psi$ be a non-zero analytic function on $\mathcal{U}$,
let $\mathcal{O}:=\{\xi\in\mathbb{R}^m:\|\xi\|^2\in\mathcal{U}\}$, and
let $\Psi(\xi):=\psi(\|\xi\|^2)\in C^\infty(\mathcal{O})$. Define a real-analytic radial conformal deformation
of the standard Euclidean metric $g_e$ by setting
$$
\mathcal{N}_\psi:=(\mathcal{O},\Psi^{-2}g_e)\,.
$$
Let $\psi_{a,b}(t):=a+bt$ define $\mathcal{N}_{a,b}$ for $(a,b)\ne(0,0)$. Although the following result is well-known,
we present a proof in Section~\ref{S4} for the sake of completeness since we will need to develop the
requisite preliminaries in any event; we suppose $m\ge3$ as that is the case of interest.

\begin{lemma}\label{L1.4} Let $m\ge3$.
\begin{enumerate}
\item If $\psi$ is linear, then $\mathcal{N}_\psi$ is a space form.
\item If $\mathcal{N}_\psi$ is a space form, then $\psi$ is linear.
\item $\mathcal{N}_{a,b}$ has constant sectional curvature
$4ab$.
\item  If
$\mathcal{M}$ has constant sectional curvature $\kappa$, $\mathcal{M}$
is locally isometric to $\mathcal{N}_{\psi_{1,4\kappa}}$.
\end{enumerate}\end{lemma}\black

\subsection{Radial conformal deformations which are centrally harmonic about an intermediate point}
Let $L_\xi$ be the second fundamental form of the geodesic sphere $S_P(\|\xi\|)$ about $P$
which passes through $\xi$.
We say $S_P(\|\xi\|)$ is {\it totally umbillic} at $\xi$ if $L_\xi$ is a multiple of the identity.
As noted by Copson and Ruse~\cite{CR40}, a radial conformal deformation of Euclidean space is in general
not centrally harmonic about any other point. {Recall that every harmonic space is Einstein and that every Einstein manifold
is real analytic.} We will prove the following result in Section~\ref{S5}.

\begin{theorem}\label{T1.5}\rm
Let $P$ be a point of an Einstein manifold $\mathcal{M}=(M,g)$ of dimension $m\ge3$. Assume that $\psi$ is real analytic and that
$\mathcal{M}_\psi$ is centrally harmonic about some vector $\xi$ with $0<\|\xi\|<\iota_P$ in geodesic coordinates.
\begin{enumerate}
\item If $ S_P(\|\xi\|)$ is not totally umbillic at $\xi$, then $\psi$ is constant.
\item If $\mathcal{M}$ is a space form,
then $\mathcal{M}_\psi$ is a space form.
\end{enumerate}\end{theorem}

\subsection{Totally umbillic geodesic spheres}
The Jacobi operator $\mathcal{J}_0(\xi)$ is a self-adjoint endomorphism of $T_PM$.
Let $\tilde{\mathcal{J}}_0(\xi)$ be the restriction of $\mathcal{J}_0$
to $\xi^\perp$, let $m_P(\xi)$ (resp. $M_P(\xi)$) be the smallest (resp. largest) eigenvalue
of $\tilde{\mathcal{J}}_0(\xi)$, and let
$$
s_P:=\inf_{|\xi|=1}\{M_P(\xi)-m_P(\xi)\}
$$
{be the minimal difference the largest and the smallest eigenvalue of $\tilde{\mathcal{J}}_0(\xi)$
for $\xi$ a unit tangent vector at $P$.}
We will establish the following result in Section~\ref{S6}.
\begin{lemma}\label{L1.6}\rm
\ \begin{enumerate}
\item $\mathcal{M}$ is a space form, then every geodesic sphere is totally umbillic.
\item If every sufficiently small geodesic sphere is totally umbillic and if $m\ge3$, then $\mathcal{M}$
is a space form.
\item If an irreducible symmetric space $\mathcal{M}$ admits a totally umbilical hypersurface $\mathcal{N}$, then
both $\mathcal{M}$ and $\mathcal{N}$ are space forms.
\item If $s_P>0$, then there exists $\varepsilon>0$ so that geodesic spheres of radius
less than $\varepsilon$ at $P$ are not totally umbillic at any point.
\item If $\mathcal{M}$ is a rank one symmetric space
or $\mathcal{M}$ is a Damek-Ricci space, and if $\mathcal{M}$ is not a space form, then $s_P>0$.
\end{enumerate}\end{lemma}

\subsection{Radial conformal deformations of the sphere}\label{S1.9}
Theorem~\ref{T1.3} and Theorem~\ref{T1.5} deal with
points within the injectivity radius. Let $\mathcal{S}:=(S^m,g_{S^m})$ where
$g_{S^m}$ is the standard round metric on the unit sphere
$S^m$ of $\mathbb{R}^{m+1}$. Denote the north and south poles of $S^m$ by $P_\pm:=(\pm1,0,\dots,0)$,
respectively; $d_{P_\pm}(\xi)=\arccos(\pm\xi^1)$ and $\iota_{\pm}=\pi$.
Let $\psi$ be a positive real analytic function of 1-variable and let
$\mathcal{S}_\psi:=(S^m,\psi((\xi^1)^2)^{-2}g_{S^m})$.
We will establish the following result in Section~\ref{S7}.

\begin{lemma}\label{L1.7}\rm $\mathcal{S}_\psi$ is centrally harmonic about the points $P_\pm$. If
$\mathcal{S}_\psi$ is not a space form and if $m\ge3$, then $\mathcal{S}_\psi$ is centrally harmonic about no
points of the sphere other than $P_\pm$.
\end{lemma}\black

\subsection{A non-flat example with trivial volume density function}
We will use Theorem~\ref{T1.3} to establish the following result in Section~\ref{S8}.
\begin{theorem}\label{T1.8}\rm If $m\ge4$ is even, then there exists a Riemannian
metric $g$ on $\mathbb{R}^m$ which is centrally harmonic about the origin, which is
not conformally flat, and which has $\Theta_0=r^{m-1}$.
\end{theorem}

\begin{remark}\rm Since the metric $g$ of Theorem~\ref{T1.8} is not conformally flat, $g$ is not flat. Since any harmonic
space with trivial volume density function is flat, $g$ is not a harmonic metric. We will show in Section~\ref{S8}
that $g$ is essentially geodesically incomplete in dimensions $4$, $6$, and $8$.
\end{remark}

\section{The proof of Lemma~\ref{L1.2}: A leading term analysis}\label{S2}
{We use Equation~(\ref {E1.c}) to assume $n\ge7$ in the proof of
Lemma~\ref{L1.2}. We express $\mathcal{H}_n(\xi)=c_n\Tr\{\mathcal{J}_{n-2}(\xi)\}+\text{lower order terms}$.
By considering product formulas, we see that the coefficients $c_n$ are dimension free so we may take $m=2$.
We set $ds^2=dr^2+f(r,\theta)d\theta^2$ where $f(r,\theta):=\{r(1+b_n(\theta)r^n)\}^2$. We then have
$\Theta(r, \theta)=r(1+{ b_n(\theta)}r^n)$ so $\mathcal{H}_n(\partial_r^\theta)={b_n(\theta)}$
where $\partial_r^\theta$ is the radial vector field pointing from the origin to $\theta\in S^1$.
We adapt an argument from
Gilkey and Park~\cite{GP20a}. Let $f_r:={\partial_r^\theta}f$, $f_{rr}={\partial_r^\theta}{\partial_r^\theta}f$, and so forth.}
We use the Koszul formula to compute:
$$\begin{array}{llll}
\Gamma_{rrr}=0,&\Gamma_{rr\theta}=0,&\Gamma_{rr}{}^r=0,&\Gamma_{rr}{}^\theta=0,\\[0.05in]
\Gamma_{r\theta r}=0,&\Gamma_{r\theta\theta}=\frac12{ f_{r}},&\Gamma_{r\theta}{}^r=0,
&\Gamma_{r\theta}{}^\theta=\frac12{ f_{r}}f^{-1},\\[0.05in]
\Gamma_{\theta\theta r}=-\frac12f_{r},&
\Gamma_{\theta\theta\theta}=\frac12f_{\theta},&\Gamma_{\theta\theta}{}^r=-\frac12f_{r},
&\Gamma_{\theta\theta}{}^\theta=\frac12{ f_{ \theta}} f^{-1}.
\end{array}$$
Thus we have that
\begin{eqnarray*}
&&\nabla_\theta \nabla_r {\partial_r^\theta}=0,\\
&& \nabla_r\nabla_\theta{\partial_r^\theta}={ \nabla_r}\{\Gamma_{\theta r}{}^\theta\partial_\theta\}=
{ \textstyle\{\frac12f_{rr}f^{-1}-\frac12f_{r}f_{r}f^{-2}+\frac14f_{r}f_{r}f^{-2}\}}\partial_\theta,\\
&&R(\partial_\theta,{\partial_r^\theta},{\partial_r^\theta},\partial_\theta)=\textstyle-\frac12f_{rr}+\frac14f_{r}f_{r}f^{-1},\\
&&\operatorname{Tr}\{{\mathcal{J}_0({\partial_r^\theta})}\}=f^{-1}\{\textstyle-\frac12f_{rr}+\frac14f_{r}f_{r}f^{-1}\}\,.
\end{eqnarray*}
We compute:
\medbreak\quad $f(r, \theta)=r^2+2{ b_n(\theta)}r^{n+2}+O(r^{n+3})$,
\smallbreak\quad $f^{-1}(r, \theta){=}r^{-2}(1-2{ b_n(\theta)}r^n+O(r^{n+1}))$,
\smallbreak\quad $-\frac12f_{ rr}=-1-{{(n+2)(n+1)}}{ b_n(\theta)}r^n+O(r^{n+1})$,
\smallbreak\quad $\frac14{ f_r^2} f^{-1}=(r+(n+2){ b_n(\theta)}r^{n+1}+O(r^{n+2}))^2r^{-2}(1-2{ b_n(\theta)}r^n+O(r^{n+1}))$
\smallbreak\quad\quad$=1+(2(n+2)-2){ b_n(\theta)}r^n+O(r^{n+1})$,
\smallbreak\quad $-\frac12f_{ rr}+\frac14{ f_r^2} f^{-1}=b_n{{(\theta)(-(n+2)(n+1)+ 2(n+1))}}r^n+O(r^{n+1})$
\smallbreak\quad\quad$=-n(n+1)b_n(\theta) r^n+O(r^{n+1}),$
\smallbreak\quad $\Tr\{\mathcal{J}{({\partial_r^\theta})}\}=f^{-1}\{-\frac12f_{ rr}+\frac14{ f_r^2} f^{-1}\}$
\smallbreak\quad\quad$=r^{-2}(1-2{ b_n(\theta)}r^n+O(r^{n+1}))(-n(n+1){ b_n(\theta)}r^n+O(r^{n+1}))$
\smallbreak\quad\quad$=-n(n+1){ b_n(\theta)}r^{n-2}+O{{(r^{n-1})}}$,
\smallbreak\quad$\nabla_{{\partial_r^\theta}}^{n-2}\Tr\{\mathcal{J}{({\partial_r^\theta})}|_{r=0}{{\}}}=-\frac{(n+1)!}{n-1}{ b_n(\theta)}$.
\medbreak\noindent Consequently,
$c_n=-\frac{n-1}{(n+1)!}$.~\qed

\section{Proof of Theorem~\ref{T1.3}: Constructing centrally harmonic spaces}\label{S3}
Let $(r,\theta)$ be geodesic polar coordinates centered at a point $P$. Choose local coordinates
$\theta=(\theta^1,\dots,\theta^{m-1})$ on
the unit sphere to express $g=dr^2+g_{ab}(r,\theta)d\theta^ad\theta^b$
and $\Theta_P(r,\theta)=\det(g_{ab}(r,\theta))^{\frac12}\nu(\theta)$ where
$d\theta=\nu(\theta)d\theta^1\cdot\cdot\cdot d\theta^{m-1}$.
Let $\rr(r)$ satisfy $\rr(0)=0$ and $d\rr=\psi(r^2)^{-1}dr$. Let $r(\rr)$ be the inverse function.
We have
\begin{eqnarray*}
g_\psi&=&\psi(r^2)^{-2}g=\psi^{-2}dr^2+\psi^{-2}g_{ab}(r,\theta)d\theta^ad\theta^b\\
&=&d\rr^2+\psi^{-2}(r(\rr)^2)g_{ab}(r(\rr),\theta)d\theta^ad\theta^b\,.
\end{eqnarray*}
Consequently, $(\rr,\theta)\rightarrow r(\rr)\cdot\theta$ gives geodesic polar coordinates for the metric $g_\psi$ and
$\rr$ is the geodesic distance function for $g_\psi$.
We then have
\begin{equation}\label{E3.a}
\Theta_{P,g_\psi}(\rr,\theta)=\psi(r(\rr)^2)^{1-m}\Theta_{P,g}(r(\rr))
\end{equation}
and
$g_\psi$ is harmonic at the point $P$ as well.~\qed

\section{Proof of Lemma~\ref{L1.4}: Space forms}\label{S4}
We adopt the following notational conventions in Section~\ref{S4}.
Let $\mathcal{M}=(M,g)$ be a Riemannian manifold and let $\mathcal{M}_\psi:=(M,\Psi^{-2}g)$ be a conformal
{radial} deformation of $\mathcal{M}$.
Let $\rho$ and $\rho_\psi$ be the Ricci tensors of $g$ and $g_\psi$. {If $\phi$ is a smooth function on $M$, let
$\Hess_g(\phi):=\nabla^2\phi$ be the Hessian of $\phi$ with respect to $g$;}
\begin{equation}\label{E4.a}
\Hess_g(\phi)=\nabla^2\phi=\{\partial_{\xi^i}\partial_{\xi^j}\phi-\Gamma_{ij}{}^k\partial_{\xi^k}\phi\}d\xi^i\otimes d\xi^j\,.
\end{equation}
Fix $\xi\in T_PM$ with $0<\|\xi\|<\iota_P$. Choose the coordinate system on $T_PM$ so $\xi=(\|\xi\|,0,\dots,0)$.
{The following is a crucial technical result that will play a central role in the proof of Lemma~\ref{L1.4} and
of Theorem~\ref{T1.5}.}

\begin{lemma}\label{L4.1}\
\begin{enumerate}
\item If $\mathcal{M}$ is centrally harmonic {about} $P$, then $\mathcal{M}$ is Einstein at $P$.
\item $\rho_{g_\psi}-\rho_{g}=\Psi^{-1}(m-2)\operatorname{Hess}_g(\Psi)
+\{-\Psi^{-1}{\Delta_{\mathcal{M}}^0}\Psi-(m-1)\Psi^{-2}\|d\Psi\|_g^2\}g$.
\item $L_\xi(\partial_{\xi^i},\partial_{\xi^j})=-\|\xi\|^{-1}\delta_{ij}+\Gamma_{ij}{}^1(\xi)$.
\item Assume $\mathcal{M}$ and $\mathcal{M}_\psi$ are Einstein at $\xi$ {and that $m\ge3$}.
\begin{enumerate}\item If $L_\xi$ is not a multiple of the identity,
then $\psi^{\prime}(\|\xi\|^2)=0$.
\item If $\psi^\prime(\|\xi\|^2)=0$, then $\psi^{\prime\prime}(\|\xi\|^2)=0$.
\end{enumerate}
\end{enumerate}
\end{lemma}

\begin{proof} If $\mathcal{M}$ is centrally harmonic {about} $P$,
then $\mathcal{H}_2$ only depends on $\|\xi\|$ so we shall write $\mathcal{H}_2(\xi)=\mathcal{H}_2(\|\xi\|)$.
 In particular, by Equation~(\ref {E1.c}), $\rho_g(\xi,\xi)=\Tr\{\mathcal{J}_g(\xi)\}$ only depends on $\|\xi\|$ so $\rho_g(\xi,\xi)=c\|\xi\|^2$
and Assertion~(1) follows. We refer to
 K\"uhnel and Rademacher~\cite{KR08} for the proof of Assertion~(2).
 If $i>1$, let
 $\sigma_i(\theta):=\|\xi\|\cos(\theta)\partial_{\xi^1}+\|\xi\|\sin(\theta)\partial_{\xi^i}$ define a curve in $S_P(\|\xi\|)$ with
 $\dot\sigma_i(0)=\|\xi\|\partial_{\xi^i}$.
 Assertion~(3) follows by polarizing the identity
\begin{eqnarray*}
L_\xi(\partial_{\xi^i},\partial_{\xi^i})&=&
\|\xi\|^{-2}\left.\left\{g(\nabla_{g,\dot\sigma_i}\dot\sigma_i,\partial_{\xi^1})\right\}\right|_{\theta=0}\\
&=&\|\xi\|^{-2}\left.\left\{(\partial_\theta^2\sigma_i,\partial_{\xi^1})+\|\xi\|^2(\nabla_{g,\partial_{\xi^i}}\partial_{\xi^i},\partial_{\xi^1})\right\}\right|_{\theta=0}\\
&=&-\|\xi\|^{-1}+\Gamma_{ii}{}^1\,.
\end{eqnarray*}

Suppose that $\mathcal{M}$ and $\mathcal{M}_\psi$ are Einstein at $\xi$.
{Since $m\ge3$, Assertion~(2) implies that $\Hess_g(\Psi)(\xi)$ is a multiple of $g$;
if $m=2$, then
we obtain no information from the Einstein condition and it is for this reason we assume $m\ge3$ henceforth
whenever using Lemma~\ref{L4.1}.}
Since $\Psi=\psi((\xi^1)^2+\dots+(\xi^m)^2)$ and we are evaluating at $\xi=(\|\xi\|,0,\dots,0)$, we use Equation~(\ref{E4.a}) to compute:
\begin{equation}\begin{array}{l}\label{E4.b}
\Hess_g(\Psi)(\xi)=\{\partial_{\xi^i}\partial_{\xi^j}\Psi
-\Gamma_{ij}{}^k\partial_{\xi^k}\Psi\}(\xi)d\xi^i\otimes d\xi^j\\[0.05in]
\quad=\{2\delta_{ij}\psi^\prime(\|\xi\|^2)+4\|\xi\|^2\delta_{1i}\delta_{1j}\psi^{\prime\prime}(\|\xi\|^2)
-2\|\xi\|\Gamma_{ij}{}^1\psi^\prime(\|\xi\|^2){\}}d\xi^i\otimes d\xi^j\\[0.05in]
\quad=(2\psi^\prime(\|\xi\|^2)+4\|\xi\|^2\psi^{\prime\prime}(\|\xi\|^2))dr\otimes dr\\[0.05in]
\qquad-2\psi^\prime(\|\xi\|^2)\ \|\xi\|\sum_{i,j>1}L(\partial_i,\partial_j)d\xi^i\otimes d\xi^j\,.
\end{array}\end{equation}
Suppose first that $L$ is not a multiple of $g$ and that $\psi^\prime(\|\xi\|^2)\ne0$.
We may then use Equation~(\ref{E4.b}) to see that $\Hess_g(\Psi)(\xi)$ is not a multiple of $g$.  Since $m\ge3$,
Assertion~2 then shows $\rho_{g_\psi}-\rho_g$ is not a multiple of $g$. This contradicts the assumption that $\mathcal{M}$
and $\mathcal{M}_\psi$ are Einstein at $\xi$ and establishes Assertion~(4a). Suppose finally that $\psi^\prime(\|\xi\|^2)=0$
and that $\psi^{\prime\prime}(\|\xi\|^2)\ne0$. Again, examining Equation~(\ref{E4.b}) shows that $\Hess_g(\Psi)(\|\xi\|^2)$ is not
a multiple of $g$ which is a contradiction; this establishes Assertion~(4b).
\end{proof}

\subsection{Analytic radial conformal deformations of $\mathbb{R}^m$}
We adopt the notation of Section~\ref{S1.6} for the remainder of this section.
Let $g_{a,b}:=(a+b\|\xi\|^2)^{-2}g_e$ on the appropriate domain for $(a,b)\ne(0,0)$.
\begin{lemma}\label{L4.2}\rm Let $c\ne 0$.
\begin{enumerate}
\item $g_{a,b}$ and $g_{b,a}$ are isometric.
\item $g_{a,b}$ and $g_{ac^{-1},bc}$ are isometric.
\item $g_{ca,cb}$ are homothetic.
\end{enumerate}
\end{lemma}

\begin{proof} Let $\eta=\|\xi\|^{-2}\xi$ for $\xi\ne0$ define {\it inversion} about the origin.
Express $\xi=r\cdot\theta$ and $\eta=t\cdot\theta$ in polar coordinates where $r=\|\xi\|$, $t=\|\eta\|$,
$\theta=\xi/\|\xi\|=\eta/\|\eta\|$, and $rt=1$. We prove Assertion~(1) by computing:
\begin{eqnarray*}
g_{a,b}&=&\frac{dr^2+r^2d\theta^2}{(a+br^2)^{2}}=\frac{t^{-4}dt^2+t^{-2}d\theta^2}{(a+bt^{-2})^{2}}
=\frac{dt^2+t^2d\theta^2}{(at^2+b)^{2}}=g_{b,a}\,.
\end{eqnarray*}
Next set $\xi=c\eta$. Since $r=ct$, we prove Assertion~(2) by computing:
\begin{eqnarray*}
g_{a,b}&=&\frac{dr^2+r^2d\theta^2}{(a+br^2)^2}=\frac{c^2dt^2+c^2t^2d\theta^2}{(a+bc^2t^2)^2}
=\frac{dt^2+t^2d\theta^2}{(ac^{-1}+bct^2)^2}=g_{ac^{-1},bc}\,.
\end{eqnarray*}
Asseretion~(3) is immediate.\end{proof}

\subsection{The proof of Lemma~\ref{L1.4}~(1)}\label{S4.1} We must show that a radial conformal deformation
of the Euclidean metric defined by a linear function is a space form. Since
$g_{1,0}=g_e$ is the Euclidean flat metric, $g_{1,0}$ is a space form metric. Stereographic projection shows that
$g_{\frac12,\frac12}=4(1+\|\xi\|^2)^{-2}g_e$ is the standard round metric on the sphere of radius 1
and hence is a space form metric. The hyperbolic metric on the unit disk is $g_{\frac12,-\frac12}=4(1-\|\xi\|^2)^{-2}g_e$
and hence is a space form metric.
Inversion about the origin, which was discussed the proof of Lemma~\ref{L4.2}, interchanges the region $0<\|\xi\|^2<1$
and $\|\xi\|^2>1$ and shows $g_{\frac12,-\frac12}$ is a space form metric on the region $\|\xi\|^2>1$ as well.
Thus $g_{\frac12,\pm\frac12}$ are space form metrics on the appropriate domains. Any metric homothetic or isometric
to a space form metric is again a space form metric. Thus Lemma~\ref{L4.2} applies to show
$g_{c^{-1}d,\pm cd}$ is a space form metric if $c\ne0$ and $d\ne0$. Set
$$c={\left|\frac ba\right|}^{1/2},\quad d=a\left|\frac ba\right|^{1/2},\quad
\varepsilon:=\operatorname{sign}\left(\frac ba\right)=\pm\,.
$$
We then have $a=c^{-1}d$ and $b=\varepsilon cd$. This shows that
$g_{a,b}=g_{c^{-1}d,\varepsilon cd}$ is a space form metric. If $a\ne0$, then $g_{a,0}$ is homothetic to the Euclidean metric and is a space form metric. Finally by Lemma~\ref{L4.2}~(1), $g_{a,0}$ and $g_{0,a}$ are isometric and hence $g_{0,a}$ is a space form metric.~\qed
\subsection{The proof of Lemma~\ref{L1.4}~(2)}
Suppose that a radial analytic conformal deformation $\mathcal{N}_\psi$
of Euclidean space is a space form and $m\ge3$. Since $g_\psi$ and $g_e$ are Einstein at any point $\xi$ in
the domain of definition and since $m\ge3$,
we may apply Lemma~\ref{L4.1} to see $\operatorname{Hess}_{g_e}(\Psi)$ is a multiple of $g_e$.
Since $\Gamma_{ij}{}^1(g_e)=0$,
Equation~(\ref{E4.b}) shows that $\psi^{\prime\prime}(\|\xi\|^2)=0$ and hence $\psi$ is linear.~\qed
\subsection{Proof of Lemma~\ref{L1.4}~(3,4)} We must show
that the metric $g_{a,b}$ has constant sectional curvature $4ab$.
The metrics $g_{a,0}$ and $g_{0,b}$ are flat and have sectional curvature $0$.
We may therefore assume $a\ne0$ and $b\ne0$.
The metrics $4(1\pm\|\xi\|^2)^{-2}ds^2_e$ have constant
sectional curvature $\pm1$, i.e. $g_{\frac12,\pm\frac12}$ has constant sectional curvature $\pm1$. Thus
Assertion~(3) holds if $(a,b)=(\frac12,\pm\frac12)$. Isometric metrics have the same sectional curvature
and thus by Lemma~\ref{L4.2}~(2), $g_{\frac12c^{-1},\pm\frac12c}$ has constant sectional curvature $\pm1$. Rescaling the metric by
a homothetic constant $d^{-2}$ rescales the sectional curvature by $d^2$. Thus
$g_{\frac12 c^{-1}d,\pm\frac 12cd}$ has constant sectional curvature $\pm d^2$. The argument of Section~\ref{S4.1} now establishes the result in general. Since
any two manifolds of the same constant sectional curvature are locally isometric, Assertion~(4) follows
from Assertion~(3).~\qed\black

\section{Proof of Theorem~\ref{T1.5}: Radial conformal deformations}\label{S5}

The notation of Equation~(\ref {E1.c}) {for the covariant deformation of the Jacobi operator}
does not distinguish between the two metrics $g$ and $g_\Psi$.
We evaluate at $\xi$. Let $\eta$, $\eta_1$, and $\eta_2$ belong to $T_\xi M$. We define
the following endomorphisms of $T_\xi M$:
\begin{eqnarray*}
&&g(\mathcal{J}_{k,g}(\eta)\eta_1,\eta_2):=\{\nabla^k_gR_g(\xi)\}(\eta_1,\eta,\eta,\eta_2;\eta,\dots,\eta),\\
&&g_\Psi(\mathcal{J}_{k,g_\Psi}(\eta)\eta_1,\eta_2)=\{\nabla_{g_\Psi}^kR_{g_\Psi}(\xi)\}(\eta_1,\eta,\eta,\eta_2;\eta,\dots,\eta)\,.
\end{eqnarray*}
We emphasize that everything is evaluated at $\xi$.  We continue our discussion.

\begin{lemma}\label{L5.1} If $\mathcal{M}$ is Einstein,  {if $m\ge3$}, if $\psi^\prime(\|\xi\|^2)=0$,
and if $\mathcal{M}_\psi$ is centrally harmonic {about} $\xi$,
then $\psi^{(k)}(\|\xi\|^2)=0$ for all $k$.
\end{lemma}

\begin{proof} By Assertion~(4a) of Lemma~\ref{L4.1}, $\psi^{\prime\prime}(\|\ \xi\|^2)=0$ as well. Suppose the Lemma is false.
Choose $n\ge2$ minimal so {that} $\psi^{(j)}(\|\xi\|^2)=0$ for $1\le j\le n$ but
{so that} $\psi^{(n+1)}(\|\xi\|^2)\ne0$. We argue for
a contradiction. Since $\mathcal{M}$ is Einstein, $\nabla_g^j\rho_g$ vanishes identically for $j\ge1$.
Since $\psi^{(j)}(\|\xi\|^2)=0$ for $1\le j\le n$, we have $\nabla_g^j=\nabla_{g_\psi}^j$ at $\xi$ for $1\le j\le n$ and we need not
distinguish the two. We may covariantly differentiate Assertion~{(2)} of Lemma~\ref{L4.1} to see
\begin{equation}\label{E5.a}
\{\nabla^j\rho_{g_\psi}\}(\xi)=\nabla^j\{\rho_{g_\psi}-\rho_g\}(\xi)=0\text{ for }j\le n-2\,.
\end{equation}
Since $\mathcal{M}_\psi$ is centrally harmonic about $\xi$, $\mathcal{H}_{g_\psi,\xi,n+1}(\eta)$
is a multiple of $\|\eta\|^{n+1}$.
Since $\mathcal{M}_\psi$ is Einstein, $\mathcal{J}_0(\eta)=c\|\xi\|^2\operatorname{id}$. We
use Equation~(\ref{E5.a}) and Assertion~(2) of Lemma~\ref{L4.1} to see that
$\mathcal{J}_j(\eta)$ is zero and hence a multiple of $\operatorname{id}$ for $1\le j\le n-2$
(if $n=2$, this assertion is vacuous).  We may therefore
use Lemma~\ref{L1.2} to see that $\Tr\{\mathcal{J}_{g_\psi,\xi,n-1}(\eta)\}$ is a multiple of
$\|\eta\|^{n-1}$.
In particular, $\Tr\{\mathcal{J}_{g_\psi,\xi,n-1}(\eta)\}=0$ if $n$ is even.

Consequently we must differentiate
the coefficients appearing in Assertion~(2) of Lemma~\ref{L4.1} to study $\nabla^{n-1}\{\rho_{g_\psi}-\rho_g\}(\xi)$.
We have
$$\begin{array}{l}
\nabla^{(n-1)}\rho_{g_\psi}(\partial_{\xi^i},\partial_{\xi^i};\partial_{\xi^i},\dots,\partial_{\xi^i})(\xi)\\[0.05in]
\qquad=\nabla^{(n-1)}\{\rho_{g_\psi}-\rho_g\}(\partial_{\xi^i},\partial_{\xi^i};\partial_{\xi^i},\dots,\partial_{\xi^i})(\xi)\\[0.05in]
\qquad=\Psi^{-1}(\xi)\left\{\begin{array}{lll}(m-1)\psi^{n+1}(\|\xi\|^2)&\text{if}&i=1\\
0&\text{if}&i>1\end{array}\right\}\,.
\end{array}$$
Since this must depend only on $\|\partial_{\xi^i}\|$, we conclude as desired $\psi^{(n+1)}(\|\xi\|^2)=0$.
\end{proof}

\noindent\subsection{The proof of Theorem~\ref{T1.5}~(1)} Let $\mathcal{M}=(M,g)$ be an Einstein manifold
{of dimension $m\ge3$}.
Suppose that $\psi$ is real analytic and that
$\mathcal{M}_\psi$ is centrally harmonic about some $\xi$ with $0<\|\xi\|<\iota_P$.
Assume the geodesic sphere about $P$ is not totally umbillic at $\xi$.
By Lemma~\ref{L4.1}~4, we
have $\psi^\prime(\|\xi\|^2)=\psi^{\prime\prime}(\|\xi\|^2)=0$. By Lemma~\ref{L5.1}, we have $\psi^{(k)}(\|\xi\|^2)=0$
for all $k\ge1$. Since $\psi$ is real analytic, this implies $\psi$ is constant.~\qed
\subsection{The proof of Theorem~\ref{T1.5}~(2)}
We may work locally and assume without loss of generality that $\mathcal{M}$ is flat space and $g=g_e$.
Suppose $\psi$
is real analytic and that $\mathcal{M}_\psi$ is centrally harmonic about some $\xi$ with $0<\|\xi\|<\iota_P$.
We assume $m\ge3$ and use Lemma~\ref{L5.1}.
\subsubsection{Suppose that $\psi(\|\xi\|^2)\ne\|\xi\|^2\psi^\prime(\|\xi\|^2)$.}
Express $\mathcal{M}_\psi=\{\mathcal{N}_{1,a}\}_{\phi_a}$
where we set $\phi_a(t):=(1+at)^{-1}\psi(t)$. We then have
$$
\phi^\prime_a(t)=\frac{-a\psi(t)+(1+at)\psi^\prime(t)}{(1+at)^2}\,.
$$
We solve the equation
$\phi^\prime_a(\|\xi\|^2)=0$ to obtain
$$
a=-\frac{\psi^\prime(\|\xi\|^2)}{\|\xi\|^2\psi^\prime(\|\xi\|^2)-\psi(\|\xi\|^2)}\,.
$$
By Lemma~{\ref{L1.4}}, $\mathcal{N}_{1,a}$ is a space form. Since $\mathcal{N}_{\phi_a}$ is centrally harmonic
about $\xi$ and $\phi_a^\prime(\|\xi\|^2)=0$, we may use Lemma~\ref{L5.1} to see that
$\phi_a^\prime(\|\xi\|^2)=0$ and hence $\phi_a$ is constant
so $\mathcal{M}$ is homothetic to $\mathcal{N}_{1,a}$ and hence is a space form.

\subsubsection{Suppose that $\psi(\|\xi\|^2)=\|\xi\|^2\psi^\prime(\|\xi\|^2)$.}
Express $\mathcal{M}_\psi=\{\mathcal{N}_{0,1}\}_\phi$ where we set $\phi(t)=t^{-1}\psi(t)$.
We then have $\phi^\prime(\|\xi\|^2)=\{t^{-2}\{-\psi(t)+t\psi^\prime(t)\}\}|_{t=\|\xi\|^2}=0$ and again we can use Lemma~\ref{L5.1}
to complete the proof.
~\qed\black

\section{The proof of Lemma~\ref{L1.6}: Totally umbillic geodesic spheres}\label{S6}
{The local isometry group of a
space form} acts transitively on the unit tangent bundle of the geodesic spheres;
consequently, the geodesic spheres in a space form are totally umbillic. This proves Assertion~(1).
We refer to Chen and Vanhecke~\cite{CV81}, Kulkarni~\cite{K75}, and
Vanhecke and Willmore~\cite{VW79} for the proof the converse assertion to establish Assertion~(2). We use Chen~\cite{C80}
to establish Assertion~(3).
Let $\sigma_{ab}(r\xi)$ be the second fundamental form of the
geodesic sphere about $P$ passing thru the point $r\xi$. Chen and Vanhecke~\cite{CV81} show
$\sigma_{ab}=r^{-1}\delta_{ab}-\frac r3R_{\xi a\xi b}(P)+O(r^2)$. Since $s_P>0$, $R_{\xi a\xi b}$ is
not a multiple of $\delta_{ab}$ and show Assertion~(4) follows.
We use Berndt, Tricerri and Vanhecke~\cite{BTV78} to derive Assertion~(5) from Assertion~(4).
The eigenvalues of the Jacobi operator are given in the first Theorem on page 96 of Section 4.2. {By hypothesis,
$\mathcal{M}$ does not have constant sectional curvature
but the remaining rank one symmetric spaces are included.} There are 6 cases in the classification (i)--(vi). In cases (i)--(v), the eigenvalues
of the Jacobi operator are $\{0,-\frac14,-1\}$ and the eigenvalue 0 appears with multiplicity 1 which
yields the eigenvalues of the reduced Jacobi operator are $\{-\frac14,-1\}$ so $M(\xi)-m(\xi)=\frac34$.
The situation in case (vi) is more complicated. Still, there is a 4-dimensional subspace where the eigenvalues
are $\{0,-\frac14,-1\}$ where $-\frac14$ has multiplicity 2. The computation of the remaining eigenvalues is
more difficult. {Nevertheless}, we obtain $M(\xi)-m(\xi)\ge\frac34$ so $s>0$ as desired.~\qed

\section{The proof of Lemma~\ref{L1.7}}\label{S7} We {adopt} the notation of Section~\ref{S1.9}.
The round sphere $\mathcal{S}$ is a space form.
Since $\mathcal{S}_\psi$ is conformally radially rotationally symmetric about the north and south poles
$P_\pm$, $\mathcal{S}_\psi$
 is centrally harmonic about these two points by Theorem~\ref{T1.3}. Suppose $\mathcal{S}_\psi$ is centrally harmonic
about some other point.
Since we are within the injectivity radius, we can apply
Theorem~\ref{T1.5} to see $\mathcal{S}_\psi$ is a space form as we have assumed $m\ge3$. This is a contradiction.~\qed

\section{The proof of Theorem~\ref{T1.8}: A non-flat example with trivial volume density function}\label{S8}
Let $m=2\mathfrak{m}\ge4$.
Let $\mathcal{M}:=(\mathbb{CP}^{\mathfrak{m}}-\mathbb{CP}^{\mathfrak{m}-1},g)$ where $g$
is the Fubini-Study metric. We have removed the cut-locus
and consequently, the underlying manifold is an open geodesic ball of radius $\frac\pi2$.
Choose $\psi$ so $\psi(r^2)^{-1}\tilde\Theta_{P,g}(r)=1$.
Then the Equation~(\ref{E3.a}) ensures $\tilde\Theta_{{{P,g_\psi}}}=1$.~\qed
\medbreak

\begin{remark}\rm We examine $\mathbb{CP}^{\frac12m}$ near the cut locus by setting set $u=\frac\pi2-r$. Set
\begin{eqnarray*}
&&\Theta(u):=\sin\left(\frac\pi2-u\right)^{{{(m-1)}}}\cos\left(\frac\pi2-u\right),\\
&&\Psi(u)=\frac{\sin\left(\frac\pi2-u\right)}{\frac\pi2-u}\cos\left(\frac\pi2-u\right)^{1/(m-1)}\,.
\end{eqnarray*}
Then $g_\psi(\partial_u,\partial_u)=\psi(u)^{-2}$ so the curves $\gamma(u)=(u,0,\dots,0)$ have
length
$$\int_{u=0}^{\frac\pi2}\psi^{-1}(u)du\,.
$$
Since $\psi(u)=\frac2\pi u^{\frac1{m-1}}+O(1)$,
the unparametrized geodesics have finite length and the resulting manifold is not geodesically complete.
We use Lemma~\ref{L4.1} to compute
\begin{eqnarray*}
&&\rho_{g}(\psi(u)\partial_u,\psi(u)\partial_u)=O(1),\\
&&\rho_{g_\psi}(\psi(u)\partial_u,\psi(u)\partial_u)=
(\rho_{g_\psi}-\rho_g)(\psi(u)\partial_u,\psi(u)\partial_u)\\&=&
(m-2)\psi(u)\psi^{\prime\prime}(u)+\psi(u)\Theta(u)^{-1}\partial_u\{\Theta(u)\psi(u)\}-(m-1)\psi^\prime(u)^2+O(1)
\end{eqnarray*}
A mathematica computation yields
$$
\rho_{g}(\psi(u)\partial_u,\psi(u)\partial_u)=\left\{\begin{array}{lll}
-\frac{28}{9\pi^2}u^{-\frac43}+O(u^{-\frac13})&\text{ if }m=4\\[0.05in]
-\frac{84}{25\pi^2}u^{-\frac85}+O(u^{-\frac35})&\text{ if }m=6\\[0.05in]
-\frac{172}{49\pi^2}u^{-12/7}+O(u^{-\frac57})&\text{ if }m=8\end{array}\right\}
$$
so this is singular at $u=0$ and $\mathcal{M}_\psi$ is essentially geodesically incomplete.
\end{remark}

\subsection*{Research support} Research partially supported by PID 2019-105138GB-C21 (Spain) and by
 the National Research Foundation of Korea (NRF) grant
funded by the Korea government (MSIT) (NRF-2019R1A2C1083957)
\end{document}